\documentclass[11pt,reqno]{amsart}

\usepackage{amsthm}
\usepackage{amsmath}
\usepackage{amssymb}
\usepackage{color}
\usepackage{enumerate}

\newtheorem{thm}{Theorem}[section]
\newtheorem{theorem}[thm]{Theorem}

\newtheorem{corollary}[thm]{Corollary}

\newtheorem{lemma}[thm]{Lemma}
\newtheorem{proposition}[thm]{Proposition}

\theoremstyle{definition}

\newtheorem{definition}[thm]{Definition}

\theoremstyle{remark}

\newtheorem{example}[thm]{Example}
\newtheorem{remark}[thm]{Remark}

\newcommand{\RR}{\mathbb R}
\newcommand{\NN}{\mathbb N}
\newcommand{\CC}{\mathbb C}

\begin{document}
\title{Necessary and sufficient conditions to perform Spectral Tetris}
\author[P.G. Casazza, A. Heinecke, K. Kornelson, Y. Wang, Z. Zhou ]{Peter G. Casazza, Andreas Heinecke, Keri Kornelson, Yang Wang, Zhengfang Zhou}
\address[P.~G.~Casazza]{Department of Mathematics, University of Missouri, Columbia, Missouri 65211, USA; E-mail: casazzap@missouri.edu}
\address[A.~Heinecke]{Department of Mathematics, University of Missouri, Columbia, Missouri 65211, USA; E-mail: ah343@mail.missouri.edu}
\address[K.~Kornelson]{Department of Mathematics, University of Oklahoma, Norman, Oklahoma 73019, USA; E-mail: kkornelson@ou.edu}
\address[Y.~Wang]{Department of Mathematics, Michigan State University, East Lansing, Michigan 48824, USA; E-mail: ywang@math.msu.edu}
\address[Z.~Zhou]{Department of Mathematics, Michigan State University, East Lansing, Michigan 48824, USA; E-mail: zfzhou@math.msu.edu}

\thanks{Casazza and Heinecke were supported by NSF DMS 1008183,
DTRA/NSF 1042701, AFOSR FA9550-11-1-0245}

\begin{abstract}
{\it Spectral Tetris} has proved to be a powerful tool for constructing \textit{sparse}
equal norm Hilbert space frames.  We introduce a new form of Spectral Tetris which
works for non-equal norm frames.   
It is known that this method cannot construct all frames --- even in the new case introduced
here.
Until now, it has been a mystery as to why Spectral Tetris sometimes works and sometimes
fails.  We will give a complete answer to this mystery by giving necessary and sufficient
conditions for Spectral Tetris to construct frames in all cases including
  equal norm frames, prescribed
norm frames,  frames with constant spectrum of the frame operator, and frames with
prescribed spectrum for the frame operator.  We present a variety of examples as well as
special cases where Spectral Tetris always works.  

\end{abstract}

\maketitle

\section{Introduction}

Hilbert space frame theory has distinguished itself for its broad application to problems in
pure mathematics, applied mathematics, computer science, engineering and much more.
Until recently, there was a serious lack of algorithms for constructing frames
with given properties, leaving the field to rely on {\it existence proofs} for the existence of
fundamental frames for applications.  So it was a major advance when in
\cite{CFMWZ09} a simple algorithm, designated {\it Spectral Tetris}, for constructing unit norm tight frames was introduced.
Since the appearence of \cite{CFMWZ09} there have been numerous papers written on variations
of Spectral Tetris. In all cases we discover that Spectral Tetris cannot construct all the frames
of a given class.  Until now, it has been a mystery when we can or cannot use Spectral Tetris
to construct families of frames.  In this paper, we introduce another version of Spectral Tetris
which for the first time can be used to construct frames with different norms for the frame
vectors and different eigenvalues for the frame operator.  
We give necessary and sufficient conditions for this version
of Spectral Tetris to work to produce the desired frames.  After this, we give necessary
and sufficient conditions for the original form of Spectral Tetris to produce the desired frames.
The main advantage here is that Spectral Tetris and
its variations give an elementary and easily implementable algorithm for constructing frames.
Although our result is reported as a theorem, it is really an algorithm for constructing certain
classes of frames via Spectral Tetris ideas.

After reviewing known results about the existence of finite frames with given properties and in particular the Spectral Tetris algorithm in Section \ref{Sec:background}, our main theorem in Section \ref{Sec:mainthm} presents a version of Spectral Tetris for the construction of sparse frames with specified frame operator and norms of the vectors. We give necessary and sufficient conditions on the prescribed sequence of eigenvalues and norms for this construction to work. We study special cases of this theorem in the remaining sections; for example, in Section \ref{less2} we classify the cases in which the original Spectral Tetris construction works to construct unit norm tight frames of redundancy less than $2$.

\section{Background}\label{Sec:background}

Let $(e_n)_{n=1}^N$ be the standard unit vector basis of $\CC^N$.
The {\it synthesis operator} of a finite sequence
$(f_m)_{m=1}^M\subseteq\CC^N$ is $F\colon\CC^M\to\CC^N$ given by
\[
Fg=\sum_{m=1}^M\langle g,e_m\rangle f_m,
\]
i.e. $F$ is the $N\times M$ matrix whose $m$-th column is the vector $f_m$. The sequence $(f_m)_{m=1}^M$ is a {\it frame} if its {\it frame operator} $S=FF^*$ satisfies $AI\leq S\leq BI$ for some positive constants $A,B$, where $I$ is the identity on $\CC^N$. In particular, the spectrum of $S$ is positive  and real.  We call $A,B$ the \textit{frame bounds} for $(f_m)_{m=1}^M$. 
The sequence is a {\it tight frame} if $A=B$, i.e. if 
\begin{equation}\label{eqn:tightframe}
Af=\sum_{m=1}^M\langle f,f_m\rangle f_m
\end{equation}
for all $f\in\mathbb{C}^N$, or equivalently if
\[
\sum_{m=1}^M\langle f_m,e_n\rangle \overline{\langle f_m,e_{n'}\rangle}=
\begin{cases}
A, & n=n',\\
0, & n\neq n'.
\end{cases}
\]
If the vectors $(f_m)_{m=1}^M$ are unit vectors which form a tight frame, the tight frame bound $A$ equals $M/N$ and is also called the {\it redundancy} of the frame. The synthesis matrix $F$ of the frame $(f_m)_{m=1}^M$ is called \emph{$s$-sparse}, if it has 
$s$ nonzero entries.

We say a frame $(f_m)_{m=1}^M$ has vectors of norms $(a_m)_{m=1}^M$ if $\|f_m\|=a_m$ for all $m=1,\ldots,M$, and call it a \emph{unit norm frame} if $\|f_m\|=1$ for all $m=1,\ldots,M$.
Unit norm tight frames provide Parseval-like decompositions in terms of nonorthogonal vectors of unit norm; Equation \eqref{eqn:tightframe} becomes \[ f=\frac{N}{M} \sum_{m=1}^M\langle f,f_m\rangle f_m, \] where each summand is a rank-one projection of $f$ onto the span of the corresponding frame vector $f_m$.

We say that a frame has a certain spectrum or certain eigenvalues if its frame operator $S$ has this spectrum or respectively these eigenvalues.  For any frame, the sum of its eigenvalues, counting multiplicities, equals the sum of the squares of the norms of its vectors: \begin{equation}\label{eqn:trace} \sum_{n=1}^N \lambda_n = \sum_{m=1}^M \|f_m\|^2. \end{equation}
This quantity will be exactly the number of vectors $M$ when we work with unit norm frames. 

Two given sequences $(a_m)_{m=1}^M$ and $(\lambda_n)_{n=1}^N$ of positive real numbers have to meet certain requirements in order for an $M$-element frame in $\CC^N$ with norms $(a_m)_{m=1}^M$ and eigenvalues  $(\lambda_n)_{n=1}^N$ to exist. Several characterizations on the sequence of norms for a tight frame to exist, are given in~\cite{CFKLT}, 
while~\cite{CL10} and ~\cite{KL04} show that an $M$-element frame in $\CC^N$ with lengths $(a_m)_{m=1}^M$ and eigenvalues  $(\lambda_n)_{n=1}^N$ exists if and only if the sequence of eigenvalues $(\lambda_n)_{n=1}^N$ \emph{majorizes} the sequence 
$(a_m^2)_{m=1}^M$. This is to say, if and only if, after arranging both sequences in decreasing order, we have
\[
\sum_{i=1}^na_i^2\leq\sum_{i=1}^n\lambda_i 
\]
for every $n=1,\ldots,N$ and then, by~\eqref{eqn:trace}, 
\[
\sum_{i=1}^Ma_i^2 = \sum_{i=1}^N\lambda_i .
\]
Building on the theory of majorization and the Schur-Horn Theorem,~\cite{CFMPS11} 
shows how to explicitly construct every possible frame whose frame operator has a given spectrum and whose vectors are of given prescribed norms. 

The results above do not address the sparsity of the frames, nor the computational complexity of the constructions. In this paper, we determine necessary and sufficient conditions on the sequences of norms and eigenvalues such that an elementary and easily implementable construction can be used to produce a highly sparse frame having the prescribed parameters.

Our algorithm is a generalization of the \emph{Spectral Tetris} algorithm, 
which was the first systematic method for constructing unit norm tight frames. It was introduced in \cite{CFMWZ09} to generate unit norm tight frames in $\CC^N$ for any dimension $N$ and any number of frame vectors $M\geq 2N$.  The algorithm was generalized in \cite{CCHKP10} to construct unit norm frames for $\mathbb{C}^N$ having a desired frame operator with eigenvalues $(\lambda_n)_{n=1}^N\subseteq[2,\infty)$ satisfying $\sum_{n=1}^N\lambda_n=M$. It is this form of the
algorithm which we review for convenience in Table~\ref{fig:originalSTC} and to which we will refer as 
\emph{Spectral Tetris} (STC).

\begin{table}[h]
\centering
\framebox{
\begin{minipage}[h]{6.0in}
\vspace*{0.3cm}
{\sc \underline{STC: Spectral Tetris Construction}}

\vspace*{0.4cm}

{\bf Parameters:}\\[-3ex]
\begin{itemize}
\item Dimension $N\in\NN$.
\item Number of frame elements $M\in\NN$.
\item Eigenvalues $(\lambda_n)_{n=1}^N \subseteq[2,\infty)$ such that
$\sum_{n=1}^N\lambda_n=M$.
\end{itemize}

{\bf Algorithm:}\\[-3ex]
\begin{itemize}
\item[1)] Set $m = 1$.
\item[2)] For $m=1,\ldots,N$ do
\item[3)] \hspace*{0.5cm} Repeat
\item[4)] \hspace*{1cm} If $\lambda_n< 1$ then
\item[5)] \hspace*{1.5cm} $f_m = \sqrt{\frac{\lambda_n}{2}} \cdot e_n+ \sqrt{1-\frac{\lambda_n}{2}} \cdot e_{n+1}$.
\item[6)] \hspace*{1.5cm} $f_{m+1} = \sqrt{\frac{\lambda_n}{2}}\cdot e_n - \sqrt{1-\frac{\lambda_n}{2}} \cdot e_{n+1}$.
\item[7)] \hspace*{1.5cm} $m = m+2$.
\item[8)] \hspace*{1.5cm} $\lambda_{n+1} = \lambda_{n+1} - (2-\lambda_n)$.
\item[9)] \hspace*{1.5cm} $\lambda_n= 0$.
\item[10)] \hspace*{1cm} else
\item[11)] \hspace*{1.5cm} $f_m = e_n$.
\item[12)] \hspace*{1.5cm} $m = m+1$.
\item[13)] \hspace*{1.5cm} $\lambda_n = \lambda_n - 1$.
\item[14)] \hspace*{1cm} end.
\item[15)] \hspace*{0.5cm} until $\lambda_n = 0$.
\item[16)] end.
\end{itemize}

{\bf Output:}\\[-3ex]
\begin{itemize}
\item Unit norm frame $(f_m)_{m=1}^M\subseteq\RR^N$.
\end{itemize}
\vspace*{0.1cm}
\end{minipage}
}
\vspace*{0.2cm}
\caption{The STC algorithm for constructing a unit norm frame with prescribed spectrum in 
$[2,\infty)$.}
\label{fig:originalSTC}
\end{table}

In this paper we entirely focus on constructions in $\RR^N$.
Our main result in is a further generalization of STC to construct frames with specified norms in $\RR^N$.  STC constructs synthesis matrices $F$ with unit norm columns of real entries, whose rows are pairwise orthogonal and
square sum to the desired eigenvalues. The resulting frame operator $FF^*$ is a diagonal matrix, having precisely the desired eigenvalues on its diagonal.  We refer the reader to~\cite{CFMWZ09} or~\cite{CCHKP10} for instructive examples on how the algorithm constructs the desired synthesis matrices by using $2\times 2$ building blocks of the form
\begin{align}\label{2by2}
\begin{bmatrix}
\sqrt{x}&\sqrt{x}\\
\sqrt{1-x}&-\sqrt{1-x}
\end{bmatrix}.
\end{align}
A version of Spectral Tetris to construct frames with any given positive spectrum, not necessarily being contained in
$[2,\infty)$ was given in \cite{CFHWZ12}. It uses modified discrete Fourier transform matrices as building blocks which might be larger than $2 \times 2$ and which result in frames of complex vectors.   

Aside from the fact that Spectral Tetris frames are easy to construct, their major advantage is the sparsity of their synthesis matrices, making them  a valuable tool in applications such as the construction of fusion frames \cite{CFHWZ12}.  In the algorithms STC and PNSTC here, all of the resulting frames are at most $2M$-sparse, i.e. each frame vector has at most $2$ nonzero coordinates.

It has been shown in \cite{CHKK10}, that tight Spectral Tetris frames are optimally sparse in the sense that given $M\geq 2N$, the synthesis matrix of the $M$-element frame for $\RR^N$ constructed via spectral tetris is sparsest within the class of all synthesis matrices of $M$-element unit norm frames for $\RR^N$.

\section{Spectral tetris for general sparse frames}\label{Sec:mainthm}

In this section, we give a generalization of Spectral Tetris which uses $2 \times 2$ blocks to construct frames with a given spectrum and vectors having a specified sequence of norms.  Frames constructed using this algorithm will be at least $2M$-sparse.   Such a construction is not always possible; we give necessary and sufficient conditions under which this generalized construction works to produce the specified frames.

\subsection{The $2 \times 2$ building blocks of Spectral Tetris.}

Spectral tetris (STC) relies on the existence of $2\times 2$ matrices $A(x)$, for given $x>0$, such that 
\begin{enumerate}[(i)]
\item $A(x)$ has orthogonal rows,
\item the columns of $A(x)$ have norm $1$,
\item  and the square of the norm of the first row is $x$.
\end{enumerate}
These properties combined are equivalent to
\[
A(x)A^*(x)=
\begin{bmatrix}
x & 0 \\
0 & 2-x
\end{bmatrix}.
\]
A matrix which satisfies these properties and which is used as a building block of the synthesis matrix constructed by STC, is
\[
A(x)=
\begin{bmatrix}
\sqrt{\frac{x}{2}} & \sqrt{\frac{x}{2}} \\
\sqrt{1-\frac{x}{2}} & -\sqrt{1-\frac{x}{2}}
\end{bmatrix}.
\]

In order to generalize Spectral Tetris to allow for varied vector norms, we modify Property (ii) above so that the columns of the $2 \times 2$ matrices must have prescribed norms $a_1$ and $a_2$.  We will see that these column lengths correspond to the respective norms of the frame vectors we are constructing. Thus, the new $2 \times 2$ building blocks, which we denote $A = A(x, a_1, a_2)$, will have the following properties:
\begin{enumerate}[(i)]
\item $A$ has orthogonal rows,
\item the columns of $A$ have norms $a_1, a_2$ respectively,
\item and the square of the norm of the first row is $x$.
\end{enumerate}
These properties are equivalent to 
\begin{equation}\label{eqn:matrix}
AA^*=
\begin{bmatrix}
x & 0 \\
0 & a_1^2+a_2^2-x
\end{bmatrix}.
\end{equation}
Such a matrix $A(x,a_1,a_2)$ exists under the conditions given in the following lemma and we explicitly construct them in the proof for later use.

\begin{lemma}\label{lem:Amatrix}  A real matrix $A = A(x, a_1, a_2)$ satisfying \eqref{eqn:matrix} exists if and only if 
\begin{enumerate}[(a)]
\item $a_1^2+a_2^2\geq x >0$, and
\item either $a_1^2,a_2^2\geq x$ or $a_1^2,a_2^2\leq x$.
\end{enumerate}
\end{lemma}

\begin{proof}  The orthogonality of the rows of $A$ holds if and only if $A$ is of the form 
\begin{equation}\label{eqn:form}
A=
\begin{bmatrix}
\alpha & \beta \\
c\beta & -c\alpha
\end{bmatrix},
\end{equation} for some nonzero $c \in \mathbb{R}$.  
If we assume that $A$ satisfies \eqref{eqn:matrix}, then the equivalent properties (i)--(iii) yield the equations: \begin{eqnarray*}  \alpha^2+\beta^2&=&x \\ \alpha^2+c^2\beta^2&=&a_1^2 \\ \beta^2+c^2\alpha^2&=&a_2^2. \end{eqnarray*}

Since either $c^2\leq 1$ or $c^2>1$, we see by the properties above that either both $a_1^2,a_2^2\geq x$ or both  $a_1^2,a_2^2\leq x$.  Moreover, 
\[ c^2\beta^2 + c^2\alpha^2 = c^2x = a_1^2 + a_2^2-x \geq 0.\]

Conversely, if we are given values for $x, a_1, a_2$ satisfying properties $(a)$ and $(b)$, we now construct a matrix $A = A(x, a_1, a_2)$ satisfying \eqref{eqn:matrix}.   To obtain orthogonal rows, we let $A$ be written as in \eqref{eqn:form} and must find values of $\alpha, \beta, c$ satisfying the equations 
\begin{itemize}
\item $\alpha^2+\beta^2=x$,
\item $\alpha^2+c^2\beta^2=a_1^2$,
\item $c^2(\alpha^2+\beta^2)=y$, where we let  $y = a_1^2 + a_2^2-x$,
\item $\beta^2+c^2\alpha^2=a_2^2$,
\end{itemize}
Note that this implies that  \[ (1+c^2)(\alpha^2+\beta^2) =a_1^2+a_2^2, \] and thus 
\[ c^2 = \frac{y}{x}. \]
There are two cases to be considered.

\emph{Case} $x=y$:
In this case, $c=1$ so we have $x=y=a_1^2=a_2^2$, so a solution for the matrix $A$ is
\begin{equation}\label{eqn:Avalues1}
A=A(x, a_1, a_2) = 
\begin{bmatrix}
\sqrt{\frac{x}{2}}& \sqrt{\frac{x}{2}}\\
\sqrt{\frac{x}{2}}& - \sqrt{\frac{x}{2}}
\end{bmatrix}.
\end{equation}

\emph{Case} $x \neq y$:
We can solve for the variable $\alpha$ when $c^2 \neq 1$:
\[ \alpha^2 = \frac{a_1^2 -xc^2}{1-c^2} = \frac{a_1^2-y}
{1-\frac{y}{x}}= \frac{x(a_1^2-y)}
{x-y}.\]
From here, we can solve for the value of $\beta$ as well:
\begin{eqnarray*} \beta^2 &=&x-\alpha^2 \\
&=& x- \frac{x(a_1^2-y)}{x-y}\\
&=& \frac{x^2-xy -xa_1^2 + xy}
{x-y} \\
&=& \frac{x(x-a_1^2)}{x-y}.
\end{eqnarray*}
Thus, we have proved that
\begin{equation}\label{eqn:Avalues2} A=A(x, a_1, a_2) =  
\begin{bmatrix}
\sqrt{\frac{x(a_1^2-y)}
{x-y}}& 
\sqrt{\frac{x(x-a_1^2)}{x-y}}\\
\sqrt{\frac{y(x-a_1^2)}{x-y}}&
- \sqrt{\frac{y(a_1^2-y)}
{x-y}}
\end{bmatrix}
\end{equation}
satisfies \eqref{eqn:matrix}.
\end{proof}

\begin{remark} 
We know that the above matrices must be governed by majorization as shown in \cite{KL04,CL10};  i.e.
to construct $A(x,a_1,a_2)$ with the desired properties, we are looking for a matrix with columns 
square summing to $a_1^2$ and $a_2^2$ and rows square
summing to $x$ and $y=a_1^2+a_2^2-x$.  In
majorization language this looks like: 
\[\begin{bmatrix}
m&n\\
p&q
\end{bmatrix}
\]
with 
\[ m^2+n^2 = x ,\mbox{ and } p^2+q^2= y,\]
and
\[ m^2+p^2 = a_1^2,\mbox{ and } n^2+q^2 = a_2^2.\]
Moreover  $a_1^2+a_2^2 = x+y$. Without loss of generality $a_1\ge a_2$.
Now the condition $a_1^2,a_2^2\le x$ in Lemma \ref{lem:Amatrix}
is equivalent to  $x\geq y$ and $a_1^2,a_2^2 \ge x$ is equivalent to $x\le y$.   
\end{remark}

\subsection{Spectral Tetris for prescribed norms}\label{subsec:algorithm}
In Table \ref{fig:PNSTC} we present a modified Spectral Tetris algorithm PNSTC, allowing the construction of frames of prescribed spectrum and non-uniform lengths. We first present an example to demonstrate the algorithm's actions and to introduce the
 \textit{cursor notation} which will be used to further describe the algorithm.

\begin{example}\label{exa:demo}
Let us construct a $6$ element frame in $\RR^4$ with eigenvalues $(\lambda_n)_{n=1}^4=(15,4,1,4)$ and norms
$(a_m)_{m=1}^6=(3,2,\sqrt{3},\sqrt{3},1,2)$.
PNSTF will provide such a frame by generating a $4\times 6$ synthesis matrix of orthogonal rows square summing to the
respective eigenvalues and columns square summing respectively to $(9,4,3,3,1,4)$.
The algorithm starts with a $4\times 6$ matrix of unknown entries and lets a cursor serve as a marker which starts at position $(1,1)$ and moves
either downward or to the right along columns and rows, assigning values to certain entries with each move.
The remaining entries are set to zero when the algorithm terminates.  At each step in the algorithm, one
of the following four cases occurs:

\emph{Case $1$}: If the cursor is at matrix location $(n,m)$ and the entries already set for row $n$ square sum to less than $\lambda_n-a_m^2$, then the current entry $(n,m)$ is set to be $a_m$ and the cursor $(n,m)$ is moved to the right,
i.e., $(n,m):=(n+1,m)$. This is, for example, the case when
the cursor is in position $(1,1)$. The matrix changes as
follows, where we denote the unknown matrix entries by $\cdot$ and
the position of the cursor by $\odot$:

\[
\begin{bmatrix}
\odot&\cdot&\cdot&\cdot&\cdot&\cdot&\\
\cdot&\cdot&\cdot&\cdot&\cdot&\cdot&\\
\cdot&\cdot&\cdot&\cdot&\cdot&\cdot&\\
\cdot&\cdot&\cdot&\cdot&\cdot&\cdot&
\end{bmatrix}
\quad \longrightarrow \quad
\begin{bmatrix}
3&\odot&\cdot&\cdot&\cdot&\cdot\\
\cdot&\cdot&\cdot&\cdot&\cdot&\cdot&\\
\cdot&\cdot&\cdot&\cdot&\cdot&\cdot&\\
\cdot&\cdot&\cdot&\cdot&\cdot&\cdot&
\end{bmatrix}.
\]

\emph{Case $2$}: If the cursor is at $(n,m)$ and the entries already set for row $n$ square sum to more then $\lambda_n-a_m^2$, then the entries $(n,m)$, $(n+1,m)$, $(n,m+1)$, and $(n+1,m+1)$ are set according to lines $10)$ and $11)$ 
or $14)$ and $15)$ of PNSTC, i.e. a $2\times 2$ block of the form in Equation (\ref{eqn:Avalues1}) or (\ref{eqn:Avalues2}) is being inserted, where $x$ is the difference between $\lambda_n$ and the square sum of the entries already placed in row $n$. The cursor is moved to $(n,m) := (n+1,m+2)$. Inserting this $2\times 2$ block makes row $n$ square sum to the desired eigenvalue $\lambda_n$ and row $n+1$ to less then the desired eigenvalue $\lambda_{n+1}$, while it makes the columns $m$ and $m+1$ have norms $a_m$ and $a_{m+1}$ as desired.
This is, for example, the case when the cursor is in position $(1,3)$.
The matrix changes as follows:

\[
\begin{bmatrix}
3&2&\odot&\cdot&\cdot&\cdot&\\
\cdot&\cdot&\cdot&\cdot&\cdot&\cdot&\\
\cdot&\cdot&\cdot&\cdot&\cdot&\cdot&\\
\cdot&\cdot&\cdot&\cdot&\cdot&\cdot&
\end{bmatrix}
\quad \longrightarrow \quad
\begin{bmatrix}
3&2&1&1&\cdot&\cdot\\
\cdot&\cdot&\sqrt{2}&-\sqrt{2}&\odot&\cdot&\\
\cdot&\cdot&\cdot&\cdot&\cdot&\cdot&\\
\cdot&\cdot&\cdot&\cdot&\cdot&\cdot&
\end{bmatrix}.
\]

\emph{Case $3$}: If, after inserting a $2\times 2$ block of the form (\ref{eqn:Avalues1}) or (\ref{eqn:Avalues2}), the cursor
is at location $(n,m)$ and the entries already assigned to row $n$ square sum to exactly $\lambda_n$, then the cursor is moved down to $(n+1,m)$. This is the case when the cursor is $(2,5)$. The matrix changes as follows:

\[
\begin{bmatrix}
3&2&1&1&\cdot&\cdot\\
\cdot&\cdot&\sqrt{2}&-\sqrt{2}&\odot&\cdot&\\
\cdot&\cdot&\cdot&\cdot&\cdot&\cdot&\\
\cdot&\cdot&\cdot&\cdot&\cdot&\cdot&
\end{bmatrix}
\quad \longrightarrow \quad
\begin{bmatrix}
3&2&1&1&\cdot&\cdot\\
\cdot&\cdot&\sqrt{2}&-\sqrt{2}&\cdot&\cdot&\\
\cdot&\cdot&\cdot&\cdot&\odot&\cdot&\\
\cdot&\cdot&\cdot&\cdot&\cdot&\cdot&
\end{bmatrix}.
\]

\emph{Case $4$}: If the cursor is at $(n,m)$ and the entries already set for row $n$ square sum to exactly 
$\lambda_{n}-a_m^2$, then the entry $(n,m)$ is set to $a_m$, and the cursor is moved down and to the right
$(n,m) :=(n+1,m+1)$.
This is, for example, the case when the cursor is $(3,5)$. The matrix changes as follows:

\[
\begin{bmatrix}
3&2&1&1&\cdot&\cdot\\
\cdot&\cdot&\sqrt{2}&-\sqrt{2}&\cdot&\cdot&\\
\cdot&\cdot&\cdot&\cdot&\odot&\cdot&\\
\cdot&\cdot&\cdot&\cdot&\cdot&\cdot&
\end{bmatrix}
\quad \longrightarrow \quad
\begin{bmatrix}
3&2&1&1&\cdot&\cdot\\
\cdot&\cdot&\sqrt{2}&-\sqrt{2}&\cdot&\cdot&\\
\cdot&\cdot&\cdot&\cdot&1&\cdot&\\
\cdot&\cdot&\cdot&\cdot&\cdot&\odot&
\end{bmatrix}.
\]

After performing all steps of PNSTC, the synthesis matrix of the desired frame is
\[\begin{bmatrix}
3&2&1&1&0&0\\
0&0&\sqrt{2}&-\sqrt{2}&0&0\\
0&0&0&0&1&0\\
0&0&0&0&0&2
\end{bmatrix}.
\]

\end{example}

Given a fixed dimension $N$ and frame cardinality $M$, a necessary and sufficient condition on the prescribed norms $(a_m)_{m=1}^M$ of the vectors and the prescribed eigenvalues $(\lambda_n)_{n=1}^N$ of the frame operator for PNSTC to work, is given in the following definition.

\begin{definition}\label{STReady}
We say two sequences $(a_m)_{m=1}^M$ and $(\lambda_n)_{n=1}^N$ are
\emph{Spectral Tetris ready} if
$\sum_{m=1}^Ma_m^2=\sum_{n=1}^N\lambda_n$
and  if there is a partition $0\leq m_1<\cdots<m_N=M$ of the set $\{0, 1, \ldots, M\}$ such that for all $k=1,\ldots,N-1$:
\begin{enumerate}[(i)]
\item $\sum_{m=1}^{m_k}a_m^2 \leq \sum_{n=1}^k\lambda_n<\sum_{m=1}^{m_k+1}a_m^2$  and
\item  if  $\sum_{m=1}^{m_k}a_m^2< \sum_{n=1}^k\lambda_n$, then $m_{k+1}-m_k\geq 2$ and 
\[
a_{m_k+2}^2\geq \sum_{n=1}^k\lambda_n-\sum_{m=1}^{m_k}a_m^2.
\]
\end{enumerate}
\end{definition}

Note that we may permute the given sequences $(a_m)_{m=1}^M$ and  $(\lambda_n)_{n=1}^N$ to make them Spectral Tetris ready, since we make no assumptions about the ordering of the sequence elements.

\begin{table}[h]
\centering
\framebox{
\begin{minipage}[h]{6.0in}
\vspace*{0.3cm}
{\sc \underline{PNSTC: Prescribed Norms Spectral Tetris Construction}}

\vspace*{0.4cm}

{\bf Parameters:}\\[-3ex]
\begin{itemize}
\item Dimension $N\in\NN$.
\item Number of frame elements $M\in\NN$.
\item Eigenvalues $(\lambda_n)_{n=1}^N$  and norms of the frame vectors $(a_m)_{m=1}^M$ such that 
$(\lambda_n)_{n=1}^N$  and $(a_m^2)_{m=1}^M$ are Spectral Tetris ready.
\end{itemize}

{\bf Algorithm:}\\[-3ex]
\begin{itemize}
\item[1)] Set $m=1$.
\item[2)] For $n=1,\ldots,N$ do
\item[3)] \hspace*{0.5cm}Repeat
\item[4)] \hspace*{1cm}If $\lambda_n\geq a_m^2$ then
\item[5)] \hspace*{1.5cm}$f_m=a_me_n$.
\item[6)] \hspace*{1.5cm}$\lambda_n=\lambda_n-a_m^2$.
\item[7)] \hspace*{1.5cm}$m=m+1$.
\item[8)] \hspace*{1cm}else
\item[9)] \hspace*{1.5cm}If $2\lambda_n=a_m^2+a_{m+1}^2$, then
\item[10)] \hspace*{2cm}$f_{m}=
\sqrt{\frac{\lambda_n}{2}}\cdot (e_n+e_{n+1}).$
\item[11)] \hspace*{2cm}$f_{m+1}=
\sqrt{\frac{\lambda_n}{2}}\cdot (e_n-e_{n+1}).$
\item[12)] \hspace*{1.5cm}else
\item[13)] \hspace*{2cm}$y=a_m^2+a_{m+1}^2-\lambda_n$.
\item[14)] \hspace*{2cm}$f_{m}=
\sqrt{\frac{\lambda_n(a_m^2-y)}{\lambda_n-y}}\cdot e_n
+\sqrt{\frac{y(\lambda_n-a_m^2)}{\lambda_n-y}}\cdot e_{n+1}.$
\item[15)] \hspace*{2cm}$f_{m+1}=
\sqrt{\frac{\lambda_n(\lambda_n-a_m^2)}{\lambda_n-y}}\cdot e_n
-\sqrt{\frac{y(a_m^2-y)}{\lambda_n-y}}\cdot e_{n+1}.$
\item[16)] \hspace*{1.5cm}end.
\item[17)] \hspace*{1.5cm}$\lambda_{n+1}=\lambda_{n+1}-(a_m^2+a_{m+1}^2-\lambda_n)$.
\item[18)] \hspace*{1.5cm}$\lambda_n=0$.
\item[19)] \hspace*{1.5cm}$m=m+2$.
\item[20)] \hspace*{1cm}end.
\item[21)] \hspace*{0.5cm}until $\lambda_n=0$.
\item[22)] end.
\end{itemize}

{\bf Output:}\\[-3ex]
\begin{itemize}
\item Frame $(f_m)_{m=1}^M\subseteq\RR^N$.
\end{itemize}
\vspace*{0.1cm}
\end{minipage}
}
\vspace*{0.2cm}
\caption{The PNSTC algorithm for constructing a frame with prescribed spectrum and norms.}
\label{fig:PNSTC}
\end{table}

\begin{definition}\label{def:complete}
The STC and PNSTC algorithms insert $2\times 2$ blocks into the synthesis matrices under construction. We call any entry of the synthesis matrix constructed by either one of these algorithms a \emph{terminal point} if it belongs to the second row of some $2\times 2$ block and an \emph{initial point}  if it belongs to the first row of a $2\times 2$ block.
Moreover, we say that a row $n$ of the synthesis matrix the algorithms construct is \textit{complete at column $m$} if the entries of row $n$ are zero for all columns to the right of column $m$, i.e. the square sums of the entries in row $n$ from columns 1 through $m$ is equal to $\lambda_n$.
\end{definition}

We now show that the properties given in Definition~\ref{STReady} are exactly the necessary and sufficient conditions which allow PNSTC to construct a frame with prescribed norms having a given spectrum.

\begin{theorem}\label{thm:norm}
Given $(a_m)_{m=1}^M\subseteq(0,\infty)$ and $(\lambda_n)_{n=1}^N\subseteq(0,\infty)$, PNSTC can be used to construct a frame
$(f_m)_{m=1}^M$ for $\mathbb{R}^N$ such that $\|f_m\|=a_m$ for $m=1,\ldots,M$ and having eigenvalues $(\lambda_n)_{n=1}^N$ if and only if there exist permutations that make the sequences
$(a_m)_{m=1}^M$ and $(\lambda_n)_{n=1}^N$ Spectral Tetris ready.
\end{theorem}

\begin{proof}
For the forward direction, assume that PNSTC will produce a frame with eigenvalues $(\lambda_n)_{n=1}^N$ and norms $(a_m)_{m=1}^M$.  The partition $\{m_k\}_{k=1}^N$ from Definition \ref{STReady}  is given by the unique coordinate  $(k, m_k)$ in each row $k$ such that no further $1 \times 1$ blocks can be inserted in row $k$.  In other words, either row $k$ is complete at column $m_k$ or row $k$ requires a $2 \times 2$ block beginning at column $m_k+1$ and will be complete at column $m_k +2$.  ( In Example \ref{exa:demo}, the coordinates $(k,m_k+1)$ are the cursor locations when we enter Cases 2, 3, or 4.)   Because the algorithm was able to complete successfully, Properties (i) and (ii) from Definition \ref{STReady} necessarily hold for each $k=1, 2, \ldots, N-1$.

Conversely, assume $(a_m)_{m=1}^M$ and $(\lambda_n)_{n=1}^N$ are Spectral Tetris ready. We must demonstrate that the PNSTC algorithm produces the $N \times M$ synthesis matrix of the frame with these corresponding properties.  It is sufficient to show that, for each $k \in \{1, 2, \ldots, N-1\}$,  if the cursor is at position $(k, m_k)$ and if a $2 \times 2$ block is needed, then the conditions of Lemma \ref{lem:Amatrix} are satisfied and therefore a $2 \times 2$ block exists that will complete row $k$ at column $m_k+2$.

Assume that for a given $k$, we have already constructed the first $m_k$ columns of the synthesis matrix, i.e. the cursor is at position $(k,m_k)$.
If $\sum_{n=1}^k\lambda_n =  \sum_{m=1}^{m_k}a_m^2$, then row $k$ is complete at column $m_k$ and we proceed
by moving the cursor to $(k+1,m_k+1)$ without inserting a $2\times 2$ block. This is done in lines $7)-8)$ of PNSTC.   Otherwise, by the definition of $m_k$,  we have $\sum_{n=1}^k\lambda_n > \sum_{m=1}^{m_k}a_m^2$ and we require a $2 \times 2$ block to complete row $k$. We show that the conditions from Lemma \ref{lem:Amatrix} are satisfied which means that the desired $2 \times 2$ block exists. In particular, we need to show that the matrix
$A=A\left(x,a_{m_k+1},a_{m_k+2}\right)$ with
\[
x = \sum_{n=1}^k\lambda_n-\sum_{m=1}^{m_k}a_m^2
\]
exists, since insertion of this block will make row $k$ square sum to the desired $\lambda_k$.
Moreover, we also have to show that the second row of $A$ square sums to at most $\lambda_{k+1}$, so that we do not reach an impediment when we construct row $k+1$.

Note that since we have assumed our sequences are Spectral Tetris ready, we have by Definition
\ref{STReady} (i)
\[
a_{m_k+1}^2\geq \sum_{n=1}^k\lambda_n-\sum_{m=1}^{m_k}a_m^2=x,  
\]
which clearly gives
\[
a_{m_k+1}^2+a_{m_k+2}^2\geq \sum_{n=1}^k\lambda_n-\sum_{m=1}^{m_k}a_m^2>0.
\]
Definition \ref{STReady} (ii) also gives
\[
a_{m_k+2}^2\geq \sum_{n=1}^k\lambda_n-\sum_{m=1}^{m_k}a_m^2=x.
\]
Thus the desired $2 \times 2$ block $A(x, a_{m_k+1}, a_{m_k+2})$ does exist by Lemma \ref{lem:Amatrix}.
It remains to show that the second row of $A$ square sums at most to $\lambda_{k+1}$, i.e. that
\[
a_{m_k+1}^2+a_{m_k+2}^2-\left(\sum_{n=1}^k\lambda_n-\sum_{m=1}^{m_k}a_m^2\right)\leq\lambda_{k+1}.
\] 
This condition, however, is equivalent to
\[
\sum_{m=1}^{m_k+2}a_m^2\leq \sum_{n=1}^{k+1}\lambda_n,
\]
which holds by Definition \ref{STReady} (i), since by Definition \ref{STReady} (ii) $m_k+2\leq m_{k+1}$.
We have shown that any time a $2 \times 2$ block is required, the conditions from Lemma \ref{lem:Amatrix} are satisfied and therefore, the required matrix exists.  This proves that PNSTC will produce a sparse frame with the given properties. 

\end{proof}

\subsection{Examples}

\begin{example}
There are choices of prescribed norms and eigenvalues which satisfy the majorization condition from \cite{CL10, KL04}, i.e. for which a frame with these given parameters exists, but for which PNSTC cannot be used to construct such a frame because no ordering of the sequence of eigenvalues and sequence of norms is Spectral Tetris ready. An example of this kind is a $4$-element $\frac{13}{3}$-tight frame in $\RR^3$ with norms $(2,2,2,1)$. \end{example}

\begin{example}  Sometimes, one ordering of the eigenvalues and norms is Spectral Tetris ready while another is not.  Given a sequence of norms $(\sqrt{3}, \sqrt{3}, 1)$ and eigenvalues $(5, 2)$, by majorization there exists a frame for $\mathbb{R}^2$ having these norms and eigenvalues. We find, however, that PNSTC cannot be performed for the sequences in the given order.  The first step would be 
\[ \begin{bmatrix} \sqrt{3}&\odot&\cdot \\ 
0&\cdot&\cdot 
\end{bmatrix}.
\]
Completing row $1$ requires a $2 \times 2$ block with columns square summing to $3$ and $1$ and rows both square summing to $2$. By Lemma \ref{lem:Amatrix} such a block does not exist --- the column square norms $3$ and $1$ are neither both greater than nor both less than  the row square norm of $2$. However,
rearranging the eigenvalues to the order $(2, 5)$ allows PNSTC to construct the desired frame. Its synthesis matrix is
\[ \begin{bmatrix} 1&1 &0 \\ \sqrt{2}&-\sqrt{2}&1 \end{bmatrix}. \]

\end{example}

\begin{example}  The ordering of the eigenvalues may not be monotone in order for Spectral Tetris to work.  Let a sequence of norms be $(a_m)_{m=1}^4 = ( \sqrt{3}, \sqrt{3}, \sqrt{2}, 1)$ and the eigenvalues be $(\lambda_n)_{n=1}^3 = (4, 3, 2)$.  Given the nonincreasing order of the norms of the vectors, the only ordering of the eigenvalues which admits Spectral Tetris is $(3, 4, 2)$.   To see this, observe that neither ordering where the eigenvalue $2$ comes first will allow for the condition $a_1^2 + a_2^2 \leq \lambda_1 + \lambda_2$ to be satisfied.  The same holds for the ordering $(3, 2, 4)$.
Moreover, both cases for which the eigenvalue $4$ comes first will require the first step in PNSTC to build the first vector to be $\sqrt{3}\cdot e_1$.  Then $a_2^2 + a_3^2 \geq \lambda_1 + \lambda_2  - a_1^2$, and PNSTC cannot proceed.  
In the ordering $(3, 4, 2)$, however, we don't find the same impediments and the synthesis matrix produced by PNSTC is
\[ \begin{bmatrix}    \sqrt{3}&0 &0 &0  \\ 0&\sqrt{3} &0 &1  \\ 0& 0&\sqrt{2} & 0 \end{bmatrix}. \]

\end{example}

\begin{example}
We construct an example where PNSTC does not work when either the eigenvalues or the norms are in monotonic order, but does have a non-monotone configuration which is Spectral Tetris ready.  We do this by combining an example for which the norms cannot be in a monotone order with an example for which the eigenvalues cannot be in a monotone order; we scale one of these to force a separation between the two systems.

It is straightforward to verify that the norms $(2,2,2,1)$ and eigenvalues $(6, 4, 3)$ are only spectral-tetris ready when placed in the order $(2, 2, 2, 1)$  and $(3, 6, 4)$.   By Example \ref{exa:permutenorms} (which we scale by 30),  the norms $(\sqrt{210}, \sqrt{210}, \sqrt{180}, \sqrt{30}, \sqrt{30})$ and eigenvalues $(220, 220, 220)$ are Spectral Tetris ready only when the norms are in one of the orderings  $(\sqrt{210}, \sqrt{180}, \sqrt{30}, \sqrt{30}, \sqrt{210})$ or  $(\sqrt{210}, \sqrt{180}, \sqrt{30}, \sqrt{210},  \sqrt{30})$.

Therefore, we find that the sequence of norms $(\sqrt{210}, \sqrt{210}, \sqrt{180}, \sqrt{30}, \sqrt{30}, 2, 2, 2, 1)$ and the eigenvalues $(220, 220, 220, 6, 4, 3)$ are not Spectral Tetris ready in any monotone ordering, but are Spectral Tetris ready in the non-monotone orderings \[ ( \sqrt{210}, \sqrt{180}, \sqrt{30}, \sqrt{30}, \sqrt{210}, 2, 2, 2, 1)\quad \textrm{and} \quad (220, 220, 220, 3, 6, 4).\]
\end{example}

\subsection{An easily-checked condition for sparse frames}

Given sequences of norms and eigenvalues, it may be time-consuming to find permutations of the prescribed eigenvalues and norms which are Spectral Tetris ready.  We present an easily-verified sufficient condition on the prescribed sequences under which PNSTC can be performed.

\begin{proposition} 
Let $(a_m)_{m=1}^M\subseteq(0,\infty)$ and $(\lambda_n)_{n=1}^N\subseteq(0,\infty)$ be increasing sequences such that  $\sum_{m=1}^Ma_m^2 = \sum_{n=1}^N \lambda_n$ and
\begin{align}\label{easy}
a_{M-2\ell}^2+a_{M-2\ell-1}^2 \le \lambda_{N-\ell} 
\end{align}
for $\ell= 0,1,\ldots,N-1$.
Then $(a_m)_{m=1}^M$ and $(\lambda_n)_{n=1}^N$ are Spectral Tetris ready, hence by Theorem \ref{thm:norm}, PNSTC can construct a frame $(f_m)_{m=1}^M$ for $\mathbb{R}^N$
with $\|f_m\|=a_m$ for $m=1,\ldots,M$ and with eigenvalues $(\lambda_n)_{n=1}^N$.
In particular, PNSTC can be performed if $a_{M}^2+a_{M-1}^2 \le \lambda_1$.
\end{proposition}

Note that the property \eqref{easy} together with $\sum_{m=1}^Ma_m^2 = \sum_{n=1}^N \lambda_n$ imply that $M \geq 2N$.

\begin{proof}
We show that we can perform PNSTC on the increasing sequences $(a_m)_{m=1}^M$ and 
$(\lambda_n)_{n=1}^N$ by verifying the sequences are Spectral Tetris ready.

Property \eqref{easy} when $\ell = N-1$ guarantees that $a_1^2 \leq a_{M-2N+1}^2 \leq \lambda_1$.  For each $k\in\{1,\ldots,N-1\}$, let $m_k$ be the unique index for which 
\begin{equation}\label{eqn:mk} \sum_{m=1}^{m_k}a_m^2 \le \sum_{n=1}^k\lambda_n < \sum_{m=1}^{m_k+1}a_m^2.\end{equation}
This provides a partition $\{m_k\}_{k=1}^N$ of $\{0, 1, \ldots, M\}$ satisfying Definition \ref{STReady}(i).  Since the norms are arranged in increasing order, 
\begin{equation}\label{eqn:mk2}  \sum_{n=1}^k \lambda_n - \sum_{m=1}^{m_k}a_m^2 \le a_{m_k+1}^2 \le a_{m_k+2}^2.\end{equation}  Then because the sequences $(a_m^2)_{m=1}^M$ and $(\lambda_n)_{n=1}^N$ have equal sum, the first inequality in \eqref{eqn:mk} gives
\begin{equation}\label{eqn:mk3} \sum_{m=m_k +1}^M a_m^2 \geq \sum_{n=k+1}^N \lambda_n .\end{equation}
For ease of notation, let $p=N-k$, so the above inequality becomes
\[\sum_{m=m_{N-p} +1}^M a_m^2 \geq \sum_{n=N-p+1}^N \lambda_n.\]
By the property in \eqref{easy} for $\ell = p-1$, we have \[\lambda_{N-(p-1)} \geq a_{M-2(p-1)}^2 + a_{M-2(p-1)-1}^2 = a_{M-2p+2}^2 + a_{M-2p+1}^2.\]   
This gives us a starting point for a lower bound in the previous inequality: 
\[\sum_{m=m_{N-p} +1}^M a_m^2 \geq \sum_{n=N-p+1}^N \lambda_n \geq \sum_{m=M-2p+1}^M a_m^2 .\]
Therefore, the starting indices in the outer sums satisfy 
\[m_{N-p}+1 \leq M-2p+1.\]
Since the sequence $(a_m)_{m=1}^M$ is increasing, we have 
\[ a_{m_{N-p}+1}^2 + a_{m_{N-p}+2}^2 \leq a_{M-2p+1}^2 + a_{M-2p+2}^2 \leq \lambda_{N-(p-1)}.\]
Translating back to using $k = N-p$, we have 
\[ a_{m_k+1}^2 + a_{m_k+2}^2 \leq \lambda_{k+1},\] which implies $m_{k+1}  \geq m_k + 2$ and together with   \eqref{eqn:mk}, confirms that the criteria in Definition \ref{STReady}(ii) are satisfied, and hence PNSTC can be performed.
\end{proof}

\section{Special Case: Unit-norm tight frames of redundancy less than $2$}\label{less2}

The original STC algorithm from \cite{CFMWZ09} is known to successfully construct unit norm tight frames of $M$ vectors in $\RR^N$, provided $M \geq 2N$.  In this section, we classify when Spectral Tetris can be used to construct unit norm tight frames when $N < M < 2N$.

\begin{theorem}\label{thm:k-inequality}
For $N < M < 2N$ and $\lambda = \frac{M}{N}$ the following are equivalent:

\begin{enumerate}[(i)]
\item STC will successfully produce a unit norm tight frame $(f_m)_{m=1}^M$
for $\RR^N$.
\item  For all $1\le k \le N-1$, if $k\lambda$ is not an integer, then we have
\begin{equation} \label{eqn:tetris}  \lfloor k\lambda \rfloor \le (k+1)\lambda -2,\end{equation} where $\lfloor x \rfloor$ is the greatest integer less than or equal to $x$.
\end{enumerate}
\end{theorem} 

\begin{proof}
If $M$ and $N$ are not relatively prime, let $P = \gcd(M,N)$, $M = PM'$ and $N = PN'$, i.e. $\gcd(M',N') = 1$.  Then STC  constructs a synthesis matrix of the block form
\begin{align*}
\begin{bmatrix}
F&  &          &  \\
  &F&          &  \\  
  &  &\ddots&  \\ 
  &  &          &F
\end{bmatrix}
\end{align*} 
composed of $P$ copies of the matrix $F$, where $F$ is the STC output synthesis matrix for $M'$ unit vectors in $N'$ dimensions.
Therefore, it is sufficient to assume that $M$ and $N$ are relatively prime. For each $k=2,\ldots,N-1$, we know that the product $k\lambda$ is not an integer and thus the $k$-th row of the synthesis matrix formed by STC must contain both the second row of a $2 \times 2$ block (terminal points) and the first row of a $2 \times 2$ block (initial points).  In other words, no row except possibly the last one can be completed with a $1 \times 1$ block.  By \cite{CFMWZ09}, Spectral Tetris can proceed from row $k$ to row $k+1$  if and only if the two terminal points contained in row $k+1$, which we denote $x_1$ and $x_2$, satisfy
\[
x_1^2+x_2^2\leq\lambda
\] 
for each $k=1, 2, \ldots N-1$.  The square sum of the entries of the first $k$ completed rows is $k\lambda$ and the square sum of the two initial points
contained in row $k$ is $k\lambda-\lfloor k\lambda \rfloor$. Since the square sum of the entries of the two columns that contain the terminal points of row $k+1$ is $2$, we therefore have
\[
x_1^2+x_2^2=2-(k\lambda-\lfloor k\lambda \rfloor).
\]
Thus Spectral Tetris works if and only if 
\[
2-(k\lambda-\lfloor k\lambda \rfloor)\leq\lambda
\]
for every $k=1,\ldots,N-1$.
\end{proof}

As it happens, Theorem \ref{thm:k-inequality} has independently been derived in~\cite{Lemvig}.
We use the condition in Theorem \ref{thm:k-inequality} to completely characterize the conditions under which Spectral Tetris will work to compute unit norm tight frames.  This condition will be completely determined by the value of the frame bound $\lambda$.  Since STC always works when $\lambda \geq 2$ and no tight frames exist when $\lambda < 1$, we only consider $1 \leq \lambda < 2$ in the following results.

Note that for a fixed positive integer $k$, the inequality $2- \frac{1}{k} \leq \lambda < 2$ is equivalent to  $\lfloor k\lambda \rfloor = 2k-1.$ Thus
\eqref{eqn:tetris} holds for a fixed value of $k$ if \[  \lambda \geq  2-\frac{1}{k+1}.\] and fails to hold for fixed $k$ if 
\[ 2-\frac{1}{k} \leq \lambda < 2-\frac{1}{k+1}.\]

\begin{theorem}\label{thm:characterization}  The algorithm STC can be performed to generate a unit norm tight frame of $M$ vectors in $\RR^N$ if and only if $\lambda := \frac{M}{N} \geq 2$ or $\lambda$ is of the form \begin{equation}\label{eqn:cutoff} \lambda = \frac{2L-1}{L} 
\end{equation}
for some positive integer $L$.  
\end{theorem}

\begin{proof}  The case $\lambda \geq 2$ is proved in \cite{CFMWZ09} so we assume $\lambda<2$.  

Let  $\lambda =  \frac{2L-1}{L}$ for some $L\in \mathbb{N}$.  The inequality \eqref{eqn:tetris} is satisfied for $k=1, 2, \ldots, L-1$, so by Theorem \ref{thm:k-inequality}, STC constructs a unit norm tight frame of $M=2L-1$ vectors for $\RR^L$.  If $N=LP$ and $M = P(2L-1)$ for some integer $P > 1$, STC also works producing a synthesis matrix composed of $L\times L$ blocks repeated $P$ times, as noted in the beginning of the proof of Theorem \ref{thm:k-inequality}.

Conversely, if $\lambda$ in reduced fractional form is $\frac{M}{N}$ with $1 \leq \lambda < 2-\frac{1}{N}$, then there is a unique value of $k\in\{1 ,\ldots,N-1\}$, for which  \[ 2-\frac{1}{k} \leq \lambda < 2-\frac{1}{k+1}.\]   The inequality \eqref{eqn:tetris} fails for this $k$, and thus by Theorem \ref{thm:k-inequality}, STC can not be performed.
\end{proof}

\begin{remark}  When $\lambda$ is written as $\frac{M}{N}$, where $M$ is the number of vectors and $N$ is the dimension of the space, it may not be in reduced form.  The condition for STC to work when $\lambda < 2$ is exactly that when the rational number $\lambda$ is reduced, it is of the form $\frac{2L-1}{L}$, for some $L\in\NN$.  If $M$ and $N$ are known to be relatively prime, this condition is equivalent to $M = 2N-1$.
\end{remark}

\begin{example}  Let the dimension $N=8$.  Then Theorem \ref{thm:characterization} clearly states that Spectral Tetris can be performed with $M \geq  15$ vectors to form unit norm tight frames.   But the theorem also states that when $M=12$ and $M=14$ we have $\lambda = \frac32$ and $\lambda = \frac74$ respectively, for which STC also works. 
\end{example}

\section{Special case: Tight frames with varied norms}

In this section, we use Theorem \ref{thm:norm} to classify the sequences of norms for which
PNSTC can construct a \textit{tight} frame of vectors having these norms.
We begin with a sufficient condition.  If the prescribed sequence of norms is $(a_m)_{m=1}^M$ and our frame vectors will reside in $\mathbb{R}^N$, we consider the case in which the tight frame bound $\lambda$ is greater than or equal to the sum of the squares of the largest two values in $(a_m)_{m=1}^M$.  We find that this condition is sufficient if the sequence $(a_m)_{m=1}^M$ is in decreasing order.

\begin{remark}  This condition is an analog to the condition of the tight frame bound being at least $2$ for the case of STC \cite{CFMWZ09}, which ensures that STC works to produce unit norm tight frames.    
\end{remark} 

\begin{theorem}\label{thm:sufficient}
Let $a_1\ge a_2 \ge \cdots \ge a_M>0$ and
\[ \lambda = \frac{1}{N}\sum_{i=1}^Ma_i^2.\]
If $a_1^2+a_2^2 \le \lambda$, then PNSTC constructs a $\lambda$-tight
frame $(f_m)_{m=1}^M$ for $\mathbb{R}^N$ satisfying $\|f_m\| = a_m$ for all $m=1,2,\ldots,M$.
\end{theorem}

\begin{proof}   We need to show that $(a_m)_{m=1}^M$ (in decreasing order) and the sequence of $N$ copies of $\lambda$ are Spectral Tetris ready.
For each $k=1, 2 \ldots N-1$, let $m_k$ be the index for which
\[
\sum_{m=1}^{m_k} a_m^2\leq k\lambda<\sum_{m=1}^{m_k+1}a_m^2.
\]
First, the hypothesis gives for any $m$, $a_m^2 + a_{m+1}^2 \leq a_1^2 + a_2^2 \leq \lambda$.  Therefore, we have \begin{equation}\label{eqn:mkplus2} \sum_{m=1}^{m_k+2} a_m^2 = \sum_{m=1}^{m_k} a_m^2 + (a_{m_k+1}^2 + a_{m_k+2}^2) \leq k\lambda + \lambda = (k+1)\lambda.\end{equation} This proves that $m_{k+1} \geq m_k+2$.  

Equation \eqref{eqn:mkplus2} gives \begin{eqnarray*} a_{m_k+2}^2 &\geq& (k+1)\lambda - a_{m_k+1}^2 - \sum_{m=1}^{m_k}a_m^2 \\ &=& k\lambda - \sum_{m-1}^{m_k}a_m^2 + (\lambda - a_{m_k+1}^2)\\ &\geq& k\lambda - \sum_{m-1}^{m_k}a_m^2, \end{eqnarray*} since $\lambda \geq a_{m_k+1}^2$.  Therefore, we have satisfied Definition \ref{STReady}.
\end{proof}

\begin{example}\label{exa:permutenorms}
As in some previous examples, we may need to permute $(a_m)_{m=1}^M$ to be able to use PNSTC. Suppose we want to construct a tight frame
in $\mathbb{R}^3$ with norms $\sqrt{7},\sqrt{7},\sqrt{6},1,1$. Then $\lambda=22/3$ and we find Spectral Tetris will not work for the ordering $(\sqrt{7},\sqrt{7},\sqrt{6},1,1)$. However the permutation
$(\sqrt{7},\sqrt{6},1,1,\sqrt{7})$ is Spectral Tetris ready, so PNSTC will work.  Note that this example also demonstrates that Theorem \ref{thm:sufficient} does not provide a necessary condition for PNSTC to work.
\end{example}

We can give a necessary and sufficient condition for a sequence of norms to yield a tight frame with PNSTC by reformulating 
Theorem \ref{thm:norm}  and Definition \ref{STReady} to the case of tight frames of vectors with non-uniform norms: 
A tight frame for $\mathbb{R}^N$ with prescribed norms $(a_m)_{m=1}^M$ having all eigenvalues equal to $\lambda=\frac{1}{N}\sum_{m=1}^{M}a_m^2$ can be constructed via PNSTC if and only if there exists an ordering of $(a_m^2)_{m=1}^M$ for which there is
a partition $0\leq m_1<\cdots<m_N=M$ of $\{0, 1, \ldots M\}$ such that for all $k = 1, 2, \ldots, N-1$:
\begin{enumerate}[(i)] 
\item  $\sum_{m=1}^{m_k}a_m^2\leq k\lambda <\sum_{m=1}^{m_k+1}a_m^2$ for all $k=1,\ldots,N-1$, and
\item if  $\sum_{m=1}^{m_k}a_m^2< k\lambda$, then $m_{k+1}-m_k\geq 2$ and 
\begin{align}\label{222}
a_{m_k+2}^2\geq k\lambda -\sum_{m=1}^{m_k}a_m^2.
\end{align}
\end{enumerate}

\section{Special Case: Frames of unit vectors}
Another special case of PNSTC is the case of unit norm but not necesarily tight frames.
Such frames are known to exist, provided  the eigenvalues of the frame sum up to the number of frame vectors  \cite{DFKLOW04}. A sufficient condition for STC to work is that, in addition, the spectrum is contained in $[2,\infty)$, as proved in \cite{CCHKP10}.    The formulation of Definition \ref{STReady} and Theorem \ref{thm:norm} for unit vectors having a prescribed sequence of eigenvalues allows for eigenvalues less than $2$.

\begin{corollary}\label{cor:unitframe}  Let $\sum_{n=1}^N \lambda_n= M$ where $M \in \mathbb{N}$ and $M \geq N$.  Then STC can be used to produce a unit norm frame for $\mathbb{R}^N$ with eigenvalues $(\lambda_n)_{n=1}^N$ if and only if there is some permutation of $(\lambda_n)_{n=1}^N$ such that there exists a partition $0\leq m_1<\cdots<m_N=M$ of $\{0,\ldots,M\}$, such that for each $k=1,\ldots,N-1$ we have 
\begin{enumerate}[(i)]
\item $m_k \leq \sum_{n=1}^k \lambda_n < m_k +1$ and 
\item if $m_k <\sum_{n=1}^k \lambda_n$, then $m_{k+1} - m_k \geq 2$.  \end{enumerate} 
\end{corollary}

This characterization provides a strict limitation on the location of eigenvalues that can be strictly less than $1$.

\begin{corollary}  If STC can be used to produce a unit norm frame for $\mathbb{R}^N$ with eigenvalues $(\lambda_n)_{n=1}^N$,
then $\lambda_k<1$ is only possible if $k=1$ or if
\begin{align*}
m_{k-1} = \sum_{n=1}^{k-1} \lambda_n.
\end{align*}
\end{corollary}

\begin{proof}
It is clear that STC can begin if $\lambda_1 <1$ by letting $m_1 = 0$ and beginning the algorithm with a $2 \times 2$ block. For each $k= 2, 3, \ldots N$, either $m_{k-1} < \sum_{n=1}^{k-1} \lambda_n$ or $m_{k-1} = \sum_{n=1}^{k-1} \lambda_n$.  In the case of the inequality, STC requires a $2 \times 2$ block with unit norm columns whose first row square sums to 
$x= \sum_{n=1}^{k-1} \lambda_n - m_{k-1}>0$  and whose second row square sums to $2-x$.  

Note that $x < 1$. Moreover, $2-x \leq \lambda_k$ by Corollary \ref{cor:unitframe}, Property (ii), hence $\lambda_{k} > 1$.  Therefore, the only eigenvalues that can be less than $1$ are $\lambda_1$ and $\lambda_{k}$ in the case where $m_{k-1} = \sum_{n=1}^{k-1} \lambda_n$. 
\end{proof}

We know from Corollary \ref{cor:unitframe} that given $(\lambda_n)_{n=1}^N$ with
$\sum_{n=1}^N\lambda_n= M > N$, STC may not be able to construct a unit norm frame with eigenvalues
$(\lambda_n)_{n=1}^N$.  
Now we will see that there is some equal-norm (but not necessarily unit-norm) frame with eigenvalues $(\lambda_n)_{n=1}^N$ which is constructible by PNSTC.

\begin{theorem}
Let $(\lambda_n)_{n=1}^N\subseteq(0,\infty)$ be decreasing. Then PNSTC can construct
an equal-norm frame for $\mathbb{R}^N$ with eigenvalues $(\lambda_n)_{n=1}^N$.
\end{theorem}

\begin{proof}
Let $\lambda = \sum_{n=1}^N \lambda_n$,
and note that 
\[ \frac{\lambda_1}{\lambda} = 1-\epsilon,\]
for some $\epsilon>0$.  
Now
choose $r\in\NN$ sufficiently large such that 
\[ \frac{r^2}{\lambda}\lambda_N \ge2 \qquad \textrm{and} \qquad r^2\epsilon \geq3.\]
Define
\[ \tilde{\lambda}_n= \frac{r^2}{\lambda} \lambda_n,\mbox{ for all } n=1,2,\ldots,N.\] 
Then $(\tilde{\lambda}_n)_{n=1}^N$ is monotone decreasing, $\sum_{n=1}^N\tilde{\lambda}_n= r^2$,
and $\tilde{\lambda}_N \ge 2$. Further, we have
\[ \left\lfloor \frac{r^2}{\lambda}\lambda_1\right\rfloor \le \frac{r^2}{\lambda}\lambda_1 = 
r^2(1-\epsilon) = r^2-r^2\epsilon
\le r^2-3.\]
By \cite[Corollary 4.9]{CCHKP10}, there exists a unit norm frame $(f_m)_{m=1}^{r^2}$ for $\mathbb{R}^N$
with $r^2$ vectors and having eigenvalues $(\tilde{\lambda}_n)_{n=1}^N$.
Hence, $(\frac{\sqrt{\lambda}}{r}f_m)_{m=1}^{r^2}$ is an equal-norm frame for $\mathbb{R}^N$
with $r^2$ vectors and eigenvalues $(\lambda_n)_{n=1}^N$.
\end{proof}

\section{Acknowledgements}

The authors wish to thank Chandra Vaidyanathan for helpful conversations and Janet Tremain for generously providing examples. The third author is grateful to the University of Missouri Department of Mathematics for their hospitality while working on this project.


\end{document}